\documentclass[12pt,leqno,a4paper]{amsart}
\usepackage{amssymb,enumerate}
\newcommand{\FF}{{\mathbb{F}}}

\newcommand{\GL}{{\operatorname{GL}}}

\newcommand{\SL}{{\operatorname{SL}}}

\newtheorem{thm}{Theorem}[section]
\newtheorem{lem}[thm]{Lemma}

\theoremstyle{definition}

\theoremstyle{remark}

\newtheorem{rems}[thm]{Remarks}

\begin{document}

\title[Adequate]
{Appendix A:  Adequacy of Representations of Finite Groups of Lie Type}

\date{\today}

\author{Robert Guralnick}
\address{Department of Mathematics, University of
  Southern California, Los Angeles, CA 90089-2532, USA}
\makeatletter\email{guralnic@usc.edu}\makeatother
 
 \thanks{Guralnick was partially supported by the NSF
  grant DMS-1001962 and  Simons Foundation fellowship 224965.}

\maketitle

\section{Adequacy}

Let $\rho:  G \rightarrow  \GL(V) = \GL(n,k)$ be a faithful irreducible representation of a group $G$
over an algebraically closed  field $k$ of characteristic $p > 0$.

We recall \cite{Thorne} that $\rho$ is called adequate if it satisfies the following conditions:

\begin{enumerate}
\item [(A1)] $p$ does not divide $\dim V$;
\item [(A2)] $H^1(G,k)=0 $; 
\item  [(A3)] $\mathrm{Ext}_G^1(V,V)=0$; and
\item  [(A4)] the linear span of $\{\rho(g) | g \in G,  \rho(g) \ \text{is semisimple} \}$ is 
$\mathrm{End}(V)$.
\end{enumerate}

This concept was introduced by Thorne as a weakening of the notion of big representations
used by Taylor and Wiles (see \cite{CHT}).  
Thorne \cite{Thorne} shows that adequacy can often allow one to prove
certain representations are automorphic or potentially automorphic.    See also \cite{luis}. 
If $\rho$ is not necessarily faithful but has a finite kernel of order not divisible by
the characteristic, then it makes no difference in any of the 4 conditions whether
we consider $V$ as $kG$-module or a $k\rho(G)$-module and we will
make no distinction.  In particular, it suffices to consider the simply connected groups
of Lie type.  

In \cite{GHTT}, it was shown  that if $G$ is finite and $p > 2 \dim V + 2$, then $\rho$ is always adequate
(actually a weaker condition suffices).  See \cite{G} for some further results on adequacy. 
In this appendix, we consider the finite groups of Lie type and show that the dimension condition
can be weakened considerably.   
 
We point out the following criterion for when a finite group of Lie type
has the same invariant subspaces as the algebraic group. 

\begin{lem} \label{lem:reduce} 
Let $G$ be a semisimple algebraic group over the algebraically closed field
$k$ of characteristic $p > 0$.  Let $F$ be an endomorphism of $G$
with $G^F$ finite.  Let $V$ be a finite dimensional rational $G$-module.
Let $T$ be a maximal torus of $T$ such that distinct weights for $T$ on $V$
remain distinct for $T^F$.   Then any $G^F$-submodule of $V$ is also
a $G$-submodule.
\end{lem}

\begin{proof}  Let $H$ be the Zariski closure of $ \langle T, G^F \rangle$.  
We claim that $H=G$.  Clearly, we may reduce to the case that $G$ is simple. 

We sketch two proofs.  The first is to note that $H$ has
connected component generated by  $T$ and some collection of root subgroups
(with respect to $T$) corresponding to a quasi-closed system of roots.  It is straightforward
to see that no such group contains $G^F$ (since the centralizer of $G^F$
in $G$ is the center of $G$,  it follows easily that the  connected component
of $H$ is semisimple of the same rank as $G$ and there is short list of possible
such subgroups -- e.g., for groups of type $A$, there are none).

Alternatively, if $W$ is any restricted module for $G$, then $G^F$ also acts irreducibly on
$W$.  In particular,  let $\mathcal{W}$ be  the collection of nontrivial irreducible composition
factors of the adjoint module (usually there is just one).   By  inspection we see that no proper
positive dimensional subgroup acts irreducibly on all $W \in \mathcal{W}$ while $G^F$ does
(see \cite{hiss} for the structure of the Lie algebra as a module for $G$ and $G^F$).

Now suppose that  $W$ is a $G^F$-invariant subspace of $V$.  In particular, it is a direct
sum of $T^F$-weight spaces, whence it is also $T$-invariant and so also invariant under
$H = G$.   
\end{proof} 

It is more convenient to deal with simply connected simple algebraic groups $G$.  All such
groups are can be defined over $\FF_p$ and we fix such an $\FF_p$-structure.  In 
particular, if $F$ is the $q$-Frobenius map on $G$, then we write $G(q)=G^F$.

We can now easily show:

\begin{thm} \label{thm: algebraic adequacy}   Let $G$ be a simply
connected simple  algebraic group over an algebraically
closed field of characteristic $p > 0$ with
$\rho:G \rightarrow \GL(V)$ with $V$ a nontrivial irreducible rational module.   
Assume that $p$ does not divide $\dim V$.
\begin{enumerate}
\item  $\rho$ is adequate; and
\item   If $q = p^a$ is sufficiently large, then $\rho:G(q) \rightarrow \GL(V)$ is adequate.
\end{enumerate}
\end{thm}

\begin{proof}  It follows by \cite[p. 182]{Jantzen} that $\mathrm{Ext}_G^1(V,V)=0$.   
Of course, $\mathrm{End}(V)$ is the linear
span of all $\rho(g)$ and since semisimple elements in $G$ are dense, it follows that (A4) holds.  By assumption
(A1) holds and since $G$ is perfect (A2) holds.

We now prove (2).   Note for 
 $a$ sufficiently large, we have that:
\begin{enumerate}[(i)]
\item $G(p^a)$ is perfect; 
\item  $V$ is irreducible for $G(p^a)$;
\item  $\mathrm{Ext}_{G(p^a)}^1(V,V)=0$ (by \cite{CPSv});  and
\item  If  $T_0$ is a   split torus of $G(p^a)$ contained
in a maximal torus $T$ of $G$, then the eigenspaces of $T$ and $T_0$ are the same on
$\mathrm{End}(V)$ (and so also on $V$). 
\end{enumerate} 
Thus, (A1), (A2) and (A3) hold for $q=p^a$.  
It follows by (1) that any $G$-submodule of $\mathrm{End}(V)$ containing
the elements of $T$ is all of $\mathrm{End}(V)$.   The eigenvalue condition tells us that
the span of the elements of $T$ is the same as $T_0$.   By Lemma \ref{lem:reduce},  
any $G$-invariant submodule of $\mathrm{End}(V)$ is also $G(q)$-invariant, whence
the the span of the $G(q)$-conjugates of $T_0$ is all of $\mathrm{End}(V)$.
\end{proof} 

If in addition we assume that all weight spaces of $T$ are $1$-dimensional, the same proof 
shows that $\rho$ is big for the algebraic group and if $a$ is sufficiently large, then
$\rho$ is big for $G(p^a)$ as well.   The proof goes through verbatim for semisimple
groups as well.    

The result also holds for the twisted finite groups with essentially the same proof.

Before we make the result more precise for $\SL_2$ (which includes the result needed in \cite{luis}), 
we digress slightly to show that one can deduce adequacy for a group $G$ from the adequacy
of a normal subgroup $H$.  First we point out a well known and elementary fact.

\begin{lem}  \label{lem:restrict}  Let $H$ be a normal subgroup of $G$.  Let
$k$ be a field and $W$ a finite dimensional $kG$-module such that 
$H^0(H,W)=H^1(H,W)=0$.  Then $H^1(G,W)=0$.
\end{lem}

\begin{proof}  This follows from the usual sequences in cohomology but we give
an elementary proof.   Suppose that $H^1(G,W) \ne 0$.  Then there exists a short
exact sequence  $0 \rightarrow W \rightarrow D \rightarrow k \rightarrow 0$
that is not split for $G$.  Since $H^1(H,W)=0$, this does split for $H$ and so
$D = W \oplus W'$ where $W'$ is a $1$-dimensional trivial $H$-module.
Since $H^0(H,W)=0$, we see that $W'=H^0(H,D)$ is unique.  Since $G$ normalizes
$H$,  $W'$ is $G$-invariant, whence $D = W \oplus W'$ as $G$-modules, a contradiction.
\end{proof}

\begin{lem} \label{lem:normal}  Let $\rho:G \rightarrow GL(n,k)= \GL(V)$ be a finite dimensional representation over the algebraically closed field $k$ of characteristic $p$. 
Let $H$ be a normal subgroup of $G$ and assume that $\rho: H \rightarrow \GL(V)$ is adequate.
If $H^1(G,k)=0$ (eg. if $G/H$ has order prime to $p$), then $\rho:G \rightarrow GL(n,k)$
is adequate.
\end{lem}

\begin{proof}  Since the representation is adequate for $H$, $p$ does not divide $\dim V$.
By assumption, $H^1(G,k)=0$.  Since (A4) holds for $H$, it obviously holds for $G$.  It remains
only to show that $\mathrm{Ext}_G^1(V,V) \cong H^1(G,V^* \otimes V)=0$.

Since $p$ does not divide $\dim V$ and since $V$ is absolutely irreducible for $H$
(and so $G$), it follows that $V^* \otimes V = W \oplus k$ where  $H^0(H,W)=0$.   Since $H^1(G,k)=0$,
it  follows by Lemma \ref{lem:restrict} that $H^1(G,W)=0$ and the result follows.
\end{proof} 

See also \cite[A.1.3]{BGG}.  

\begin{thm} \label{thm:sl2}   Let $p$ be a prime,  $k$ an algebraically closed field of characteristic $p$
and $V(1)$ the natural $2$-dimensional module for $\SL_2(k)$.  If $a < p$, let $V(a)$ denote that $a$th
symmetric power of $V(1)$ (in particular,  $V(a)$ is irreducible of dimension $a+1$).
If $a < p -1$, then $V:=V(a)$ is adequate for $\SL_2(p^b)$ for any $b > 1$.
\end{thm}

\begin{proof}  By assumption $\dim V= a +1 < p$.  
By \cite{AJL}, $\mathrm{Ext}_{\SL_2(p^b)}^1(V,V)=0$.  Since
$\SL_2(p^b)$ is perfect, $H^1(\SL_2(p^b),k)=0$.   
Let $T_0$ be a maximal split torus of $\SL_2(p^b)$
with $T$ a maximal torus of $\SL_2(k)$ containing $T_0$.  It is straightforward to see 
that distinct $T$-eigenspaces
on $V \otimes V$ are also distinct $T_0$ eigenspaces.   Let $X$ the linear span of all
$\SL_2(p^b)$ conjugates of $T_0$ in $\mathrm{End}(V)$.   
By Lemma \ref{lem:reduce}, $X$  is $\SL_2(k)$ invariant.  
Since the linear span of the elements of
$T$ is the same as that of $T_0$ (because of the eigenspace condition),  
it follows  that $X$ contains the linear span
of all semisimple elements of $\SL_2(k)$ and so $X=\mathrm{End}(V)$.  
This verifies the last condition needed for adequacy.
\end{proof}

\begin{rems}  
\begin{enumerate}
\item  If   $p=7$ and $a=5$, a computer calculation (done by Frank L\"ubeck)
shows that  (A4)  holds for $\SL_2(7)$ acting on $V(5)$.  The first two 
conditions are clear and the third follows
by \cite{AJL}, whence $V(5)$  is adequate for $\SL_2(7)$.  
\item If $p=11$ and $a=5$, then in fact   $\mathrm{Ext}_{\SL_2(11)}^1(V,V)$  
is $1$-dimensional \cite{AJL}.   Thus,  $V(5)$ is not adequate for $\SL_2(11)$.   By the theorem, 
it is adequate for $\SL_2(11^b)$ for any $b > 1$. 
\item  By Lemma \ref{lem:normal}, if $V(a)$ is adequate for $G:=\SL_2(p^b)$, it is adequate
for any overgroup of $G$ in $\GL(V(a))$ normalizing $G$.  
\item  It does seem plausible that the linear span of the semisimple elements of
$\SL_2(p)$ is all of $\mathrm{End}(V)$ for any restricted $\SL_2(p)$-module.
Indeed, we know of no example of an irreducible module $V$ in characteristic $p$  of dimension not a multiple of 
$p$  for a finite simple group $G$ where the $p'$-elements do not span $\mathrm{End}(V)$.  In fact, the only
example we know is for $G={^2}F_4(2)'$ with $\dim V = 2048$ and $p=2$.  In this case, the number of
odd order  elements in $G$ is less than $(\dim V)^2$.    
\item Since typically (but not always), $\mathrm{Ext}_S^1(V,V)=0$ for 
$S$ a quasi-simple finite group of Lie type with
$V$ an irreducible module in the natural characteristic, it seems as though very 
many absolutely irreducible
$S$-modules of dimension not a multiple of $p$ are adequate. 
\item  There are many examples constructed of absolutely irreducible modules for finite groups where
(A4) fails (and indeed (A1), (A2) and (A3) hold) but aside from the case $G={^2}F_4(2)'$, the modules are all induced modules.   
See \cite{G}. 
\end{enumerate} 
\end{rems}

One can prove \cite[Cor. 2.5.4]{CHT} on big representations
of $\SL_2$ with essentially the  same proof as above. 

\begin{lem}    Let $p$ be a prime,  $k$ an algebraically closed field of characteristic $p$
and $V(1)$ the natural $2$-dimensional module for $\SL_2(k)$.  If $a < p$, let $V(a)$ denote that $a$th
symmetric power of $V(1)$ (in particular,  $V(a)$ is irreducible of dimension $a+1$).
\begin{enumerate}
\item   If $p >  2a + 3$, then $V(a)$ is big for $\SL_2(p)$.
\item  If $p > a +1$, then $V(a)$ is big for $\SL_2(p^b)$ for any $b > 1$.
\end{enumerate}
\end{lem}

\begin{proof}  It follows by \cite{AJL} that for $G=\SL_2(p^b)$
under the conditions described above we have $\mathrm{Ext}_G^1(V(a), V(a))=0$.
Under the hypotheses above, all weight space of a maximal torus $T$ of $G$
are $1$-dimensional and that the weight spaces on $V \otimes V$ are the same
as for a maximal torus $S\ge T$ of the algebraic group.   Now argue as above.
\end{proof}

If $p$ is large, we can prove a somewhat more precise result. 
We only prove this for the split form of $G$.  One could prove a variant
for the other forms.

\begin{thm} \label{thm: p large}  Let $G$ be a simple algebraic group over the 
algebraically closed field $k$ of characteristic $p \ge  4(h-1)$
where $h$ is the Coxeter number of $G$.    
Let $r$  be a positive integer.  Let $\lambda$ be a dominant weight for $G$  that is $p^r$ restricted.
 Let $V:=V(\lambda)$ be the irreducible $G$-module
with high weight $\lambda$.   If  $s  > r$, then $V$ is adequate for $G(p^s)$ if and only 
if $p$ does not divide $\dim V$.
\end{thm}

\begin{proof}    It suffices to assume that $p$ does not divide 
$\dim V$ and show that $V$ is adequate for $G(p^s)$.
There is no loss of generality in assuming that $G$ is simply connected. 

By \cite[Thm. 3.4]{BNP1},   $\mathrm{Ext}_{G(p^s)}^1(V,V)=0$. 
By \cite[Lem. 2.1(d)]{PSS}, it follows that every composition factor of 
$V \otimes V^*$ is $p^{r+1}$-restricted (and so 
in particular is irreducible for $G(p^s)$).  

Let $T_0$ be a maximal split torus of $G(p^s)$ with $T$ a maximal torus of $G$ containing $T_0$. 
We  claim that for any two distinct weights $\alpha, \beta$ of $T$ on  $V \otimes V^*$,  $\alpha - \beta$ is 
not $p^{r+1}-1$ times a weight.   This follows by \cite[Lem. 2.1]{PSS}  since $p \ge 4(h-1)$.
Thus,  $\alpha \ne \beta$ on $T_0$.   

By Lemma \ref{lem:reduce}, it follows that $G(p^s)$ and $G$ have the same invariant subspaces
on $\mathrm{End}(V) = V^* \otimes V$.   Since the linear span of the elements of $T_0$ is the same
as for $T$, it follows that $\mathrm{End}(V)$ is the linear span of the $p'$-elements in $G(p^s)$.
Thus (A4) holds.   Since $G(p^s)$ is perfect, the result holds.
\end{proof}  

Note that since the Coxeter number of $\SL_2$ is $2$,  Theorem \ref{thm: p large} includes 
Theorem \ref{thm:sl2}.

We thank Frank L\"ubeck for his computer calculations and 
Dan Nakano and Len Scott for very helpful comments.   We also
thank Florian Herzig and Richard Taylor for comments on earlier drafts
of the paper.

\end{document}